\documentclass[11pt]{article}
\usepackage[utf8]{inputenc}
\usepackage{amsmath,amsthm,amssymb}
\usepackage{tikz-cd} 
\usepackage{mathtools} 
\usepackage{textcomp} 
\usepackage{bbm} 
\usepackage{proof} 
\usepackage[tiny]{titlesec} 
\usepackage{hyperref}
\hypersetup{
	colorlinks=true,
	linkcolor=blue,
}

\usepackage{cleveref} 

\newtheorem{theorem}{Theorem}[section] 

\newtheorem{lemma}[theorem]{Lemma}
\newtheorem{proposition}[theorem]{Proposition}
\newtheorem{corollary}[theorem]{Corollary}
\newtheorem{observation}[theorem]{Observation}

\newtheorem{axiom}[theorem]{Axiom}

\theoremstyle{definition}
\newtheorem{definition}[theorem]{Definition}

\theoremstyle{remark}
\newtheorem{remark}[theorem]{Remark}

\newcommand{\N}{\mathbb{N}} 

\newcommand{\set}{\mathbf{Set}} 
\newcommand{\gpd}{\mathbf{Gpd}} 
\newcommand{\sset}{\mathbf{sSet}} 
\newcommand{\bisim}{\mathbf{ssSet}} 

\newcommand{\catC}{\mathcal{C}} 
\newcommand{\catM}{\mathcal{M}} 

\newcommand{\catE}{\mathcal{E}} 

 \newcommand{\refl}{\textup{\textsf{refl}}} 
 \newcommand{\transport}{\textup{\textsf{transport}}} 

\newcommand{\universetyp}{\mathcal{U}} 
\newcommand{\pointtyp}{\textup{\textbf{1}}} 

\newcommand{\issmall}{\textup{\textsf{issmall}}} 
\newcommand{\isequiv}{\textsf{\textup{IsEquiv}}} 

\newcommand{\tope}{\textup{\textsf{tope}}}
\newcommand{\shape}{\textup{\textsf{shape}}}
\newcommand{\type}{\textup{\textsf{type}}}
\newcommand{\context}{\textup{\textsf{context}}}

\newcommand{\simplex}{\Delta^0} 
\newcommand{\simplext}{\Delta^1} 
\newcommand{\simplextt}{\Delta^2} 

\newcommand{\horn}{\Lambda_1^2} 
\newcommand{\bsimplext}{\partial\simplext} 
\newcommand{\bsimplextt}{\partial\simplextt} 

\newcommand{\homtyp}{\textup{\textsf{hom}}} 
\newcommand{\idx}{\textup{\textsf{id}}} 

\newcommand{\trans}{\textup{\textsf{trans}}} 

\newcommand{\dua}{\textup{\textsf{dua}}} 
\newcommand{\arrtofun}{\textup{\textsf{arrtofun}}} 
\newcommand{\idtoiso}{\textup{\textsf{idtoiso}}} 
\newcommand{\rezk}{\textup{\textsf{rezk}}} 
\newcommand{\yon}{\textup{\textsf{yon}}} 
\newcommand{\evid}{\textup{\textsf{evid}}} 

\newcommand{\cocone}{\textup{\textsf{cocone}}} 
\newcommand{\cone}{\textup{\textsf{cone}}} 

\newcommand{\ie}{\textit{i.e. }}

\title{Limits and colimits of synthetic $\infty$-categories}
\date{}
\author{C\'esar Bardomiano Mart\'inez \\
	Department of Mathematics and Statistics\\
	University of Ottawa \\
	\textit{cbard035@uottawa.ca}}

\begin{document}
	
	\maketitle
	
	\begin{abstract}
	We develop the theory of limits and colimits in $\infty$-categories within the synthetic framework of simplicial Homotopy Type Theory developed by Riehl and Shulman. We also show that in this setting, the limit of a family of spaces can be computed as a dependent product.	
	\end{abstract}
	

\section{Introduction}

Defining the notion of $\infty$-category is always difficult a task. There have been many different proposed definitions, notably to mention a few, by Boardman and Vogt \cite{boardmanvogt} with the idea of quasi-categories, later exploited by Joyal \cite{joyal2008notes} and Lurie \cite{lurie}, or by Rezk in \cite{rezk} in the form of \textit{complete Segal spaces}. In the case of $(\infty,1)$-categories, most of known definitions have been shown to be equivalent in an appropriate homotopy theoretic sense, see for example the work by To\"en \cite{Ton2005VersUA}, Bergner \cite{bergner2007three} and Joyal-Tierney \cite{joyaltierney}.

Recent work by Riehl and Shulman \cite{riehlshulman} develops a \textit{synthetic theory of $(\infty,1)$-categories}. They achieve this by implementing an extension of Homotopy Type Theory. Moreover, Riehl explores in \cite{riehlundergrads} the novelty and advantages that support the philosophy behind this approach. We highlight some of this points here. Firstly, all construction are by design invariant under equivalence. Another aspect of this synthetic theory is that by construction compatible with the Univalence Axiom. Moreover, the standard model of the theory, the category of bisimplicial sets, has a well-known homotopy theory. Furthermore, it is possible to carry this synthetic theory internally to any Grothendieck $\infty$-topos (see \Cref{inftytoposinterpretation}). Finally, another pleasant consequence of this synthetic account is that somewhat simplifies some definitions and constructions, and therefore lightens the development of the theory of $(\infty,1)$-categories. To begin with, the synthetic definition of $(\infty,1)$-category is fairly succinct in comparison with any other given before.

Previously, Voevodsky's simplicial model for HoTT \cite{lumsdainekapulkin} interpreted types as $\infty$-groupoids. However, we do not have a similar interpretation of types as $\infty$-categories, where by $\infty$-category we mean $(\infty,1)$-category. To overcome such difficulty \cite{riehlshulman} introduce a new type theory by adding to HoTT a strict interval object and the novel idea, due to Shulman and Lumsdaine (unpublished), of a new type former called \textit{extension type}. We call this \textit{simplicial Homotopy Type Theory} or sHoTT for short. This type theory can be seen as an instance of a cubical type theory, however, different to that presented in \cite{Cohen2015CubicalTT}.

In this setting, it is possible to define \textit{Rezk types} as certain special types which play the role of $ \infty $-categories. In this synthetic treatment they retrieve some expected results from higher categories, including a theory for adjunctions, the analog of left (right) fibrations, which are called \textit{(contravariant) covariant families}, and the Yoneda lemma for these fibrations. Additional work by Buchholtz and Weinberger \cite{buchwein} studies \textit{synthetic cocartesian fibrations} as a generalization of covariant families. Further treatment can be found in Weinberger's PhD thesis \cite{weinthesis}, this includes \textit{two-sided fibrations}.

Exploiting results from \cite{shulmanreedy} it is shown in \cite{riehlshulman} that sHoTT has a model in bisimplicial sets where Rezk types correspond to complete Segal spaces (also called Rezk spaces).

However, we are not in the position to claim that sHoTT is indeed able to capture all category theory of $ (\infty,1) $-categories. For example, the construction of the opposite of an $ (\infty,1) $-category is not implemented yet, and we lack of a Yoneda embedding in which an $ (\infty,1) $-category is embedded into an $ (\infty,1) $-category of presheaves of spaces. Another important problem is that the type of \textit{discrete types} is a \textit{Segal type} that does not coincide with the $ \infty $-category of $ \infty $-groupoids in the bisimplicial sets model.

Despite the current inherent limitations of the type theory, we have tried to explore which other categorical properties can be obtained with the theory as is. Throughout this paper we intend to present a reasonable theory of limits and colimits of diagrams of $(\infty,1)$-categories. In this setting, we can prove most of the expected properties that are familiar from category theory. Under the mild assumption of the existence of a type-theoretic universe of spaces, we can compute the limit of a space-valued diagram as a dependent product (see \Cref{directedunivalence:remark}).

\subsection{Limits and colimits}

Limits and colimits have been studied extensively for quasi-categories in \cite{lurie}, and for Segal spaces. A short presentation appears in \cite{yonedarasekh}. We introduce the definitions of limits and colimits within this synthetic theory and verify that they are consistent with the current definition for Rezk spaces in the bisimplicial model of sHoTT. We also prove some expected results like \Cref{univpropcolim}, which can be phrased as the universal property of colimits. We also present the interaction between limits and colimits with \textit{adjoints}, \Cref{rapl} shows that ``right adjoints preserve limits.''

We prove in \Cref{uniquelimits} that in a Rezk type the \textit{type of all limits} of a given diagram is a proposition. This can be understood as a uniqueness property. Lastly, in \Cref{sectioncomputation} we show that in an appropriate sense any limit of spaces can be computed simply as a dependent product. The goal is to replicate the fact that for any diagram of spaces $ G: I \to \infty\text{-}\gpd$ where $ I $ is a set, $ \lim_I G = \prod_{i\in I}G_i $. The difficulty carrying this computation in sHoTT is that we are unable to construct the correct type of spaces (discrete types) within the theory. We implement this using a directed formulation of the Univalence Axiom due to Cavallo, Riehl, and Sattler which allows to assume the existence of a type with the desired properties.

The reader will appreciate that the study of limits and colimits in the synthetic setting is relatively simple. Our only prerequisite is some knowledge of Homotopy Type Theory. This is in line with the goal of having a synthetic theory of $\infty$-categories where these are somewhat basic objects out of which we can effortlessly develop a robust synthetic theory.

\subsection{Outline}

In \Cref{sectionstt} we start by giving an introduction to the work
of Riehl-Shulman. The material we present here is not exhaustive, but for the
understanding and development of this work it will be enough. All results
contained in this section are due to Riehl-Shulman and can be found in
\cite{riehlshulman}. The reader interested in the details of simplicial Homotopy Type Theory
is invited to read the just mentioned reference, the experienced reader may skip this whole section.

Having the basic theory at hand we define in \Cref{sectiondefinitions} limits
and colimits. We characterize them in \Cref{univpropcolim}. This result can be understood as their universal property. We also prove the analogous result from category theory that right adjoints preserve limits and left adjoints
preserve colimits.

In \Cref{sectionrezk} we study limits in Rezk types. This special case yields \Cref{uniquelimits} which is the uniqueness of limits up to equality. Up to this point all the results and proofs can be dualized.
Finally, in \Cref{sectioncomputation} we carry out the computation of the limit of a
diagram of ``spaces'' as a dependent product. Here we use \textit{univalent covariant families}, due to Cavallo, Riehl, and Sattler (unpublished) to make sense of the $\infty$-category of spaces, as simply taking the type of all Rezk types does not give the correct object.

Finally, in \Cref{semantics} we verify that our definitions are consistent with the semantics. The general procedure we follow is to first interpret our type-theoretic definition in the intended semantics of bisimplicial sets, and then prove that the resulting statement is equivalent to the existing definition. For limits and colimits we do this in \Cref{semanticscomparisonlc}. 

\paragraph{Computer verified proofs.}

While preparing this paper, a formalization project for sHoTT \cite{rzk} came to light. The overall goal is to, within the present framework, obtain computer verified proofs of results about $ \infty $-categories. Using the new proof assistant \texttt{Rzk}, the project aims to formalize the content of \cite{riehlshulman}, \cite{buchwein}, and the content of the present work.\\

\begin{par}
	\textbf{Acknowledgement.} This work is part of the author's ongoing PhD thesis under the direction of Simon Henry. The author wants to thank his supervisor for his insights, comments and suggestions which greatly improved early versions of this paper.\\
	The author also acknowledges the support of the Natural Sciences and Engineering Research Council of Canada (NSERC), under the grant, reference number RGPIN-2020-06779, awarded to Simon Henry. The author is grateful for the support granted by the Department of Mathematics and Statistics of the University of Ottawa.
\end{par}


\section{Extension types, Segal types and covariant families}\label{sectionstt}

In this section we introduce the necessary material about the type theory that we need in the paper. We refrain from providing detailed proofs. The best reference for this is the original paper \cite{riehlshulman} where simplicial Homotopy Type Theory was first introduced. We indicate precisely where each result can be found in \textit{Ibidem}.

\subsection{Extension types} \label{extensiontypes}

Essentially, simplicial Homotopy Type Theory is ordinary Homotopy Type Theory augmented with some axioms postulating the existence of a strict interval object $\mathbbm{2}$. That is, a type $\mathbbm{2}$ equipped with a total order, a smallest element $0$ and a distinct largest element $1$. In a general type $X$, an arrow in $X$ can be defined as a map $f: \mathbbm{2} \to X$, with the source and target of $f$ being the image of the endpoints $0, 1 : \mathbbm{2}$. However, one would like to be able to talk about the family of hom types from $x$ to $y$, for $x,y : X$. This would not be achievable in ordinary type theory, it can be done thanks to a new type former called an \textit{extension type}. However, the inner workings of extension types force us to single out the interval $\mathbbm{2}$ by putting it in a separate layer of a layered type theory.

More concretely, simplicial type theory is built as a three layer type theory: the layer of \textbf{cubes}, the layer of \textbf{topes}, and the layer of \textbf{types}. All rules for each layer can be found in greater detail in our main reference \cite{riehlshulman}, here we just give a small account together with some of the fundamental results we will need.

The first layer is a type theory with finite products of types. Some rules of this type theory include the following:
\[
\infer{\pointtyp \textup{ cube}}{
}
\quad\quad
\infer{\Gamma \vdash\star:\pointtyp}{
}
\]

\[
\infer{I\times J \textup{ cube}.}{
	I \, \textup{cube} & J \, \textup{cube}
}
\]

The second layer is an intuitionistic logic over the layer of cubes. Its types are called \textbf{topes}. Topes can be regarded as polytopes embedded in a cube. They admit operations of finite conjunction and disjunction, but negation, implication or quantifiers are not part of this theory. The complete rules for topes can be found in \cite[Figure 2]{riehlshulman}. There is a ``tope equality'' that appears in the third layer as a ``strict equality.'' For example, this tope equality is used to define some \textit{shapes} below, and it is denoted with the symbol ``$\equiv$.''

In simplicial type theory, the \textbf{cubes} layer is given as follows: It starts by postulating the cube $\mathbbm{2}$, which has two elements $ 0:\mathbbm{2} $ and $ 1:\mathbbm{2} $. This cube also comes with a inequality tope
\[
\infer{(x\leq y)\,\tope}{
	x:\mathbbm{2} & y:\mathbbm{2}.
}
\]
In addition, there are axioms that turn the inequality tope into a total order relation on $\mathbbm{2}$ with distinct endpoints $0$ and $1$. Recall that these axioms are:

\[
\infer{\vdash x\leq x}{
	x:\mathbbm{2}
}\quad\quad
\infer{(x\leq y),(y\leq z) \vdash (x\leq z)}{
	x:\mathbbm{2},y:\mathbbm{2},z:\mathbbm{2} 
}\quad\quad
\infer{(x\leq y),(y\leq x) \vdash (x\equiv y)}{
	x:\mathbbm{2},y:\mathbbm{2}
}
\]
\[
\infer{\vdash (x\leq y)\vee (y\leq x)}{
	x:\mathbbm{2},y:\mathbbm{2}
}\quad\quad
\infer{\vdash (0\leq x)}{
	x:\mathbbm{2}
}\quad\quad
\infer{\vdash (x\leq 1)}{
	x:\mathbbm{2}
}\quad\quad
\infer{\vdash \bot}{
	(0\equiv 1)
}.
\]

The rest of the cubes are generated by finite products of the cube $\mathbbm{2}$. Using cubes and topes we can introduce \textbf{shapes} as follows:

\[
\infer{\{t:I\: | \:\phi\}\quad \shape.}{
	I\quad \textup{\textsf{ cube}} & t:I\vdash \phi\quad \tope
}
\]

Some relevant shapes are:\\

$ \simplex:\equiv \{t:\pointtyp|\top \},$

$ \simplext:\equiv \{t:\mathbbm{2}|\top \},$

$ \simplextt:\equiv \{\langle t_1,t_2\rangle:\mathbbm{2}\times\mathbbm{2}|t_2\leq t_1 \},$

$ \bsimplext:\equiv \{t:\mathbbm{2}|(t\equiv 0)\vee (t\equiv 1)\},$

$ \bsimplextt:\equiv \{\langle t_1,t_2\rangle:\simplextt| (0\equiv t_2\leq t_1)\vee (t_1\equiv t_2)\vee ( t_2\leq t_1 \equiv 1)\},$

$ \horn:\equiv \{\langle t_1,t_2\rangle:\simplextt| (t_1\equiv 1)\vee (t_2\equiv 0)\}.$\\

More can be said about this strict interval: The category of simplicial sets is the classifying topos for such strict intervals \cite{moerdijkmaclane}. Furthermore, as explained in \cite{riehlshulman}, by embedding simplicial sets into the category of bisimplicial sets it is possible to show that the later presents the ``classifying ($ \infty$,1)-topos'' of strict intervals.

Finally, there is a third layer of types that has all the ordinary type formers of Homotopy Type Theory and one additional type former, the \textit{extension type}, that involves the previous layers. Its formation rule is:
\[
\infer{ \Xi \:| \: \Phi \: | \: \Gamma\vdash \left\langle \prod_{t:I \:| \:\psi}A(t)\middle|_{a}^{\phi} \right\rangle \,\type. }{
	\deduce{\Xi \:| \:\Phi\vdash \Gamma\, \context}{\{t:I \:| \:\phi\}\,\shape} & \deduce{\Xi, t:I \:| \:\Phi,\psi \:| \:\Gamma \vdash A\,\type}{\{t:I \:| \:\psi\}\,\shape}
	& \deduce{\Xi, t:I \:| \:\Phi,\phi \:| \:\Gamma \vdash a:A}{t:I \:| \:\phi\vdash \psi}
}
\]
We refer to \cite{riehlshulman} for the precise formulation of its rules. The name extension types is suggestive on how to construe them. We can think of $\{t:I \:| \:\phi\}$ as a ``sub-shape'' of $\{t:I\: | \:\psi\}$ and read the judgment $\Xi, t:I \,\, | \,\, \Phi,\phi\: | \:\Gamma \vdash a:A$
as a function $\phi\rightarrow A$. We could represent a point in an extension type with a dashed arrow in the commutative diagram:

\begin{center}
	\begin{tikzcd}[sep=small]
		\phi \ar[r] \ar[d,hookrightarrow] & A \\
		\psi. \ar[ru,dashrightarrow] &
	\end{tikzcd}
\end{center}

This dashed arrow does not have to be unique in any sense. The benefit of this diagrammatic representation of extension types will be obvious later on when we introduce \textit{Segal types}. Note that the $\homtyp$ type in \Cref{homtypedefinition} is a special case of an extension type. As noted in \cite{buchwein}, it is conceivable to obtain a similar theory by considering HoTT plus a strict interval theory, and replace extension types using identity types (see \cite[Pag. 119]{riehlshulman} and \cite[Section 2.4]{buchwein}), but this makes the theory harder to use. In this paper we will follow the original Riehl-Shulman formulation of sHoTT using extension types.

\begin{remark} \label{cubical:remark}
  There is a connection of sHoTT with cubical type theory \cite{Cohen2015CubicalTT}. We refer to \cite[Remark 3.2]{riehlshulman} where this is explained further. We just mention that the cubical path-type $\textsf{\textup{Path}}_A(x,y)$ can be expressed as the extension type $\langle \prod_{t:\mathbb{I}}A |_{\textup{rec}_\vee(x,y)}^{t \equiv 0 \vee t\equiv 1} \rangle$. This is analogous to the $\homtyp$ types we introduce below in \Cref{homtypedefinition}.
\end{remark}

Extension types behave very much like dependent product types. In particular, one can think of extension types as $\prod $-types with a judgmental equality added to it. Furthermore, many results which hold for $\prod$-types are also valid for extension types. For example:

\begin{theorem} \label{theorem43}
	If $t:I\: | \:\phi\vdash\psi$, $X:\{t:I\: | \:\psi\}\rightarrow \mathcal{U}$ and $Y:\prod\limits_{t:I\: | \:\psi}(X\rightarrow U)$, while $a:\prod_{t:I\: | \:\psi}X(t)$
	and $b:\prod_{t:I\: | \:\phi}Y(t,x(t))$, then
	\begin{center}
		$\left\langle \prod_{t:I\: | \:\psi}(\sum_{x:X(t)}Y(t,x))\middle|_{\lambda t.(a(t),b(t))}^{\phi}\right\rangle \simeq \left(\sum_{f:\left\langle\prod_{t:I\: | \:\psi}X(t)\middle|_a^{\phi}\right\rangle}\left\langle\prod_{t:I\: | \:\psi}Y(t,f(t))\middle|_f^{\psi}\right\rangle\right)$.
	\end{center}
\end{theorem}
\begin{proof}
	This is \cite[Theorem 4.3]{riehlshulman}.
\end{proof}

This is similar to \cite[Theorem 2.15.7]{hottbook}. The functions needed for the proof come from introduction and computation rules for extension types, these rules are quite similar to those of dependent products. We will need this idea of the proof later in \Cref{isoequivdepsum}.
Also, we have:

\begin{theorem}(\cite[Theorem 4.4]{riehlshulman})\label{theorem44}
	Suppose that we have $t:I\: | \:\phi\vdash\psi$, $t:I\: | \:\psi\vdash\chi$,
	$X:\{t:I\: | \:\chi\}\to \universetyp$ and $a:\prod\limits_{t:I\: | \:\phi}X(t)$.
	Then
	
	\[\left\langle\prod\limits_{t:I\: | \:\chi}X\middle|_{a}^{\phi}\right\rangle \simeq \sum_{f:\left\langle\prod_{t:I\: | \:\psi}X\middle|_a^{\phi}\right\rangle}\left\langle\prod_{t:I\: | \:\chi}X\middle|_f^{\psi}\right\rangle.\]
	
      \end{theorem}

      \subsubsection{Relative function extentionality axiom} \label{extensionality}

      In Homotopy Type Theory we use the function extensionality axiom for dependent functions. We recall this formulation.
      \begin{axiom}
      Let $B:A \to \universetyp$ be a dependent type family over a type $A$. For $f,g: \prod_{a:A}B(a)$, the canonical map $f=g \to \prod_{a:A}(fa=ga)$ is an equivalence.  
      \end{axiom}
       Note that the equality on the left is in the dependent product. An element of $\prod_{a:A}(fa=ga)$ is usually refer to as a homotopy, and often this type is written as $f \sim g$. There is a version of function extensionality for extension types.
      \begin{axiom} \label{extensionextensionality}
        Suppose $t:I \; | \; \phi \vdash \psi $ and that $A:\{t:I \; | \; \psi \} \to \universetyp$ is such that each $A(t)$ is contractible, and moreover there is an element $a: \prod_{t:I \; | \; \phi}A(t) $, then $\left\langle \prod_{t:I \; | \; \psi}A(t)|_a^\phi \right\rangle $ is contractible.
      \end{axiom}

      Further discussion and different formulations on this axiom can be found in \cite[4.4]{riehlshulman}. Moreover, taking $\phi$ to be $\bot$ in \Cref{extensionextensionality} gives us function extensionality. In the paper, we assume \Cref{extensionextensionality}, hence we also have function extensionality, and we use the notation $f\sim g$ from above.

\subsection{Segal types}

With extension types and simplices at hand it is feasible to introduce ``morphisms'' of a type as ``points'' in a ``hom space.'' The coherences that
witness ``composition'' between these morphisms are elements in another suitable type. We can use the above to define \textit{Segal types}.

Let us recall first that from the rules presented in \cite[Figure 3]{riehlshulman} it is possible to prove the following: Given a type $A$, in order to construct an element $a:A$ in context $\bsimplext$ we can just give two elements $x,y:A$, and vice versa, any two elements $x,y:A$ determine another element $a:A$ in context $\bsimplext$. Similarly, the construction of an element $a:A$ in context $\bsimplextt$ is equivalent to give three terms $a_0,a_1,a_2:A$ in context $t:\mathbbm{2}$ such that $a_0[0/t]\equiv a_1[0/t],$ $a_0[1/t]\equiv a_2[0/t]$ and $a_1[1/t]\equiv a_2[1/t]$. For more details on this we refer to \cite[Section 3.2]{riehlshulman}. The last construction can be seen as the ``boundary of a triangle in $A$'', see \Cref{typcomposition} below.

\begin{definition} \label{homtypedefinition}
	Given $x,y:A$, determining a term $[x,y]:A$ in context $\partial\simplext$, define
	\begin{center}
		$\homtyp_A(x,y):\equiv \left\langle \simplext\rightarrow A\middle|_{[x,y]}^{\partial\simplext}\right\rangle$.
	\end{center}
	A term in this type is called an \textbf{arrow} in $A$.
\end{definition}

Note that this is just an extension type where the type family has constant value $A$. An element $f:\homtyp_A(x,y)$ has the property that $f(0)\equiv x$ and $f(1)\equiv y$. Thus, such an element captures the idea of what an arrow in $A$ from $x$ to $y$ should be.

\begin{definition}\label{typcomposition}
	For $x,y,z:A$ and $f:\homtyp_A(x,y),g:\homtyp_A(y,z)$ and $h:\homtyp_A(x,z)$ there is a term $[x,y,z,f,g,h]:A$ in context $\partial\simplextt$, define
	\begin{center}
		$\homtyp_A^2\left(\begin{tikzcd}[sep=tiny]
			& y \ar[dr,dash,"g"] & \\
			x \ar[ur,dash,"f"] \ar[rr,dash,"h"'] & & z \\
			& \text{} & 
		\end{tikzcd}\right):\equiv \left\langle\simplextt\rightarrow A\middle|_{[x,y,z,f,g,h]}^{\partial\simplextt} \right\rangle$.
	\end{center} 
\end{definition}

For an interpretation of the types above we use the bisimplicial sets model of sHoTT. We give more details in \Cref{semanticscomparisonlc}.

\begin{definition} \label{segaldefi}
	A \textbf{Segal type} is a type $A$ such that for all $x,y,z:A$ and $f:\homtyp_A(x,y),g:\homtyp_A(y,z)$ the type
	\begin{center}
		$\sum\limits_{h:\homtyp_A(x,z)}\homtyp_A^2\left( \begin{tikzcd}[sep=tiny]
			& y \ar[dr,dash,"g"] & \\
			x \ar[ur,dash,"f"] \ar[rr,dash,"h"'] & & z \\
			 & \text{} &
		\end{tikzcd}\right)$
	\end{center}
	is contractible.
\end{definition}

Informally, $A$ is a Segal type if any pair of composable arrows have an essentially unique composite.

The first component of the center of contraction of this type is the \textbf{composition} of $f$ and $g$, we adopt the usual notation $g\circ f$ for such term. This composition of arrows is associative whenever it is defined. For any type $ A $ we can define $\idx_x:\homtyp_A(x,x)$ as $\idx_x(t)\equiv x$ for all $t:\simplext$. Furthermore, under the assumption that $A$ is a Segal type this element  is the \textbf{identity} for the composition (see \cite[Section 5.2 and 5.3]{riehlshulman}). We can also get a useful characterization of Segal types:

\begin{theorem}(\cite[Theorem 5.5]{riehlshulman}) \label{rmkcorollary56}
	A type $A$ is a Segal type if and only if the restriction map
	\[
	(\simplextt \to A) \to (\horn \to A)
	\]
	is an equivalence.
\end{theorem}

Moreover, this allows to show that if $A:X\rightarrow \universetyp$ is a family over a type or shape $X$ such that for all $x:X$ the type $A(x)$ is  a Segal type then the dependent product of the family $A$ over $X$ is again a Segal type \cite[Corollary 5.6]{riehlshulman}. Plainly this means that the dependent product of Segal types is again a Segal type.

Let $A,B$ be types and $\phi:A\rightarrow B$ a function. For any $x,y:A$ there is an induced function
\begin{center}
	$\phi_{\#}:\hom_A(x,y)\rightarrow \hom_B(\phi(x),\phi(y))$
\end{center}
defined by precomposition. We use $\phi$ for this induced function instead of $\phi_{\#}$ when there is no risk of confusion. We have:

\begin{proposition}\label{proposition61}
	Any function $\phi : A \rightarrow B$ between Segal types preserves identities and composition. This is, for all $a,b,c,d:A$ and $f:\homtyp_A(a,b),\, g:\homtyp_A(b,c)$ we have $\phi(g\circ f)=\phi(g)\circ \phi(f)$ and $\phi(\idx_d)=\idx_{\phi(d)}$.
\end{proposition}
\begin{proof}
	See \cite[Proposition 6.1]{riehlshulman}.
\end{proof}

This proposition makes patent the functorial behaviour of functions between Segal types. From \Cref{rmkcorollary56} the type
of functions $A\rightarrow B$ is a Segal type when $B$ is itself a Segal type and $A$ is a shape or type. Given two functions $f,g:A\rightarrow B$ we can select
a term $\alpha:\homtyp_{A\rightarrow B}(f,g)$, spelling out the definition of such term we get for all $s:\mathbbm{2}$ a function
$\alpha(s):A\rightarrow B$ such that $\alpha(0)\equiv f$ and $\alpha(1)\equiv g$. Additionally, for all $a:A$ we have 
\[ \lambda (s:\mathbbm{2}).\alpha(s)(a):\homtyp_B(fa,ga).
\]
We denote this function as $\alpha_a$ and we refer to it as the component of $\alpha$ at $a:A$. An element $\alpha:\homtyp_{A\rightarrow B}(f,g)$ is called
\textbf{natural transformation} between $f$ and $g$. Such terminology is supported by the following remarks:

\begin{remark}\label{proposition65}
	Let $A$ be a type or shape, $B$ a Segal type and $f,g:A\rightarrow B$. Then the induced function
	\[ \homtyp_{A\rightarrow B}(f,g)\rightarrow \prod\limits_{a:A}\homtyp_B(fa,ga)\]
	is an equivalence \cite[Proposition 6.3]{riehlshulman}. And also \cite[Proposition 6.5]{riehlshulman}, if $\alpha:\homtyp_{A\rightarrow B}(f,g),\beta:\homtyp_{A\rightarrow B}(g,h)$
	and $x:A$ then
	\[(\beta\circ\alpha)_x=\beta_x\circ\alpha_x\text{ and } (\idx_f)_x=\idx_{f(x)}.\]
	
\end{remark}

This is to say that natural transformations are component-wise defined and their composition is performed point-wise.

\begin{remark}\label{proposition66}
	Under the same assumptions of the previous remark, if $\alpha:\homtyp_{A\rightarrow B}(f,g)$ and $k:\homtyp_{A}(x,y)$ then $\alpha_y\circ fk=gk\circ\alpha_x$.
\end{remark}

Hence we can say that a natural transformation is truly natural in the categorical sense.

\subsubsection{Discrete types}

\begin{definition}{\cite[Definition 7.1]{riehlshulman}} \label{discretetypes}
	For any type $ A $, there is map
	\[ \mathsf{idtoarr}:\prod_{x,y:A} (x=_{A}y) \to \homtyp_{A}(x,y) \]
	given by path induction $ \mathsf{idtoarr}_{x,x}(\refl_x):\equiv \idx_x $. We say that the type $ A $ is discrete if $ \mathsf{idtoarr}_{x,y} $ an equivalence for all $ x,y:A $.
\end{definition}

The first aspect to note about discrete types is that they are also Segal types \cite[Proposition 7.3]{riehlshulman}.

\subsubsection{Rezk types}

In the category of bisimplicial sets, Segal spaces are, intuitively, bisimplicial sets containing categorical and homotopical information. This data is not necessarily compatible. Complete Segal spaces fix this problem. More precisely, adopting the terminology and notation of \cite{rezk}, we have $X_{\text{hoequiv}}$ as the \textit{space of homotopy equivalences} of $X$. This space $X_{\text{hoequiv}}$ is contained in $X_1$ and consist of \textit{homotopy equivalences}. Then there is an obvious map $u:X_0 \to X_{\text{hoequiv}}$, if $ u $ is an equivalence of spaces we say that the Segal space is \textbf{complete}. The map $u$ is not always a weak equivalence, an example of such Segal space can be found in \cite[Example 2.57]{rasekh2}. A similar phenomenon occurs for Segal types. These have an associative, unital composition but are not ``univalent'' just as Segal spaces are not necessarily complete. To amend this problem and obtain ``better behaved'' types we need to introduce \textit{Rezk types}.

Let $A$ be a Segal type and $f:\homtyp_A(x,y)$, we define the type
\[\textsf{\textup{isiso}}(f):\equiv \left( \sum\limits_{g:\homtyp_A(y,x)}g\circ f=\idx_x \right) \times \left( \sum\limits_{h:\homtyp_A(y,x)}f\circ h=\idx_y \right)\]
and say that $f$ is an \textbf{isomorphism} if this type is inhabited. The reader might note that we do not ask for the left and right inverse of an isomorphism to be equal. The reason behind this can be found in \cite[Section 10.1]{riehlshulman} There is an immediate characterization of isomorphisms:

\begin{proposition}\label{proposition101}
	Let $A$ be a Segal type and $f : \homtyp_A (x, y)$. Then $f$ is an isomorphism if and only if we have $g : \homtyp_A (y, x)$ with $g \circ f = \textup{\textsf{id}}_x$ and 
	$f \circ g = \textup{\textsf{id}}_y$.
\end{proposition}
\begin{proof}
	This is \cite[Proposition 10.1]{riehlshulman}.
\end{proof}

The type $\textsf{\textup{isiso}}(f)$ is a proposition \cite[Proposition 10.2]{riehlshulman}, hence we can define the type of isomorphisms between any
two terms $x,y$ in a Segal type $A$ as
\[x\cong y:\equiv \sum\limits_{f:\homtyp_A(x,y)}\textsf{\textup{isiso}}(f).\]

Given a Segal type $A$ and $x:A$ is clear that $\idx_x$ is an isomorphism. We get a family of functions
\[\idtoiso_A:\prod\limits_{x,y:A} \left( x=y\rightarrow x\cong y \right)\]
constructed by path induction via  $\idtoiso_A(\refl_x):\equiv \idx_x$. Although the function $\idtoiso_A$ is dependent of the type $A$, we will drop $A$ from the notation since it is always clear over which type the function is defined.

\begin{definition}
	A Segal type $A$ is a \textbf{Rezk type} if $\idtoiso$ is an equivalence. We denote the inverse by $ \rezk $.
\end{definition}
We will need the following fact about Rezk types from \cite[Proposition 10.9]{riehlshulman}

\begin{proposition} \label{exponentialrezk}
	If $ B $ is a Rezk type then for any type or shape $ X $, the type $ B^X $ is a Rezk type.
\end{proposition}

\subsection{Covariant families}

Left fibrations were introduced by Joyal in \cite{joyal2008volii}, and further studied in \cite{lurie}. These are the quasi-categorical version of functors cofibered in groupoids. In turn, the category of bisimplicial sets also have a well-known theory of left fibrations (see \cite{yonedarasekh}). In sHoTT we study them as
\textit{covariant families}, and they will be an important tool in later sections.

Let  $C:A\rightarrow \universetyp$ be a type family over a type $A$. Two elements $x,y:A$, $f:\homtyp_a(x,y)$ and $u:C(x),v:C(y)$ define the type
\[\homtyp_{C(f)}(u,v):\equiv \left\langle \prod\limits_{t:\mathbbm{2}}C(f(t))\middle|_{[u,v]}^{\partial\simplext}\right\rangle.\]
This type can be thought of as the type of arrows from $u$ to $v$ over $f$ in the total type of $C$.

\begin{definition} \label{covardefi}
	A type family $C:A\rightarrow \universetyp$ over a Segal type $A$ is \textbf{covariant} if for any $f:\homtyp_A(x,y)$ and
	$u: C(x)$ the type
	\[
	\sum\limits_{v:C(y)}\homtyp_{C(f)}(u,v)
	\]
	is contractible. Similarly, if for any $f:\homtyp_A(x,y)$ and $v: C(y)$ the type
	\[
	\sum\limits_{u:C(x)}\homtyp_{C(f)}(u,v)
	\]
	is contractible we say that $C$ is \textbf{contravariant}.
\end{definition}

In the situation of the definition above, we denote the center of contraction of each type as $(f_*u,\trans_{f,u})$ and $(f^*v,\trans_{f,v})$, respectively.

From \cite[Remark 8.3]{riehlshulman} we have that if $g: B \rightarrow A$ is a function and $C : A \rightarrow\mathcal{U}$ is a covariant type family, then
$\lambda b. C(g(b)) : B \rightarrow \mathcal{U}$ is also covariant. This just means that a covariant family is stable under substitution.

The first important example of a covariant family is given by the $\homtyp$ type (see \cite[Proposition 8.13]{riehlshulman}):

\begin{proposition}\label{proposition813}
	Let $A$ be a type and fix $a : A$. Then the type family $\lambda x. \homtyp_A (a, x) : A \rightarrow \mathcal{U}$ is covariant if and only if $A$ is a Segal type.
      \end{proposition}
We can show that if $C:A\rightarrow \universetyp $ is a covariant family over a Segal type $A$ then the total type of $C$
\[
\sum\limits_{x:A}C(x)
\]
is a Segal type \cite[Theorem 8.8]{riehlshulman}. Also, as previously advertised, if the covariant family is defined over a Segal type then it is functorial in the following sense:
\begin{proposition}\label{proposition816}
	Suppose $A$ is a Segal type and $C : A \rightarrow \mathcal{U}$ is covariant. Then given $f : \homtyp_A (x, y)$, $g : \homtyp_A (y, z)$, and $u : C(x)$, we have
	$g_* (f_* u) = (g\circ f )_* u$ and $(\textup{\textsf{id}}_x )_*u = u$.
\end{proposition}
\begin{proof}
	This is \cite[Proposition 8.16]{riehlshulman}.
\end{proof}
Furthermore, naturality holds:
\begin{proposition}\label{proposition817}
	Let $C, D : A \rightarrow\universetyp $ be two covariant families and a fiber-wise map $\phi:\prod\limits_{x:A}C(x)\rightarrow D(x)$.
	Then for any $f:\homtyp_A(x,y)$ and $u : C(x)$,
	\[
	\phi_y (f_* u) = f_* (\phi_x(u)).
	\]
\end{proposition}
\begin{proof}
	See \cite[Proposition 8.17]{riehlshulman}.
\end{proof}

If we have a covariant family $C:A\rightarrow \universetyp$ over a Segal type, a path $p:x=y$ and $u:C(x)$ then the element
$(\idtoiso(p))_*u:C(y)$ can be understood as transporting $u$ along the arrow $\idtoiso(p)$. The covariant transport and the usual type-theoretic transport along paths coincide, by this we mean that they are propositional equal:
\begin{lemma}\label{lemma107}
	If $A$ is Segal and $C : A\rightarrow \mathcal{U}$ is covariant, while $e : x =_A y$, then for any $u : C(x)$ we have
	\begin{center}
		$\textup{\textsf{idtoiso}}(e)_* u = \textup{\textsf{transport}}^C (e, u)$.
	\end{center}
\end{lemma}
\begin{proof}
  \cite[Lemma 10.7]{riehlshulman}.
\end{proof}

\subsubsection{Yoneda lemma}

The Yoneda lemma is with no doubt a fundamental result of category theory. In the context of $\infty$-categories it has been proved for different models, such as, for quasi-categories in \cite{lurie}, \cite{joyal2008notes}, for Segal spaces in \cite{yonedarasekh} and $\infty$-cosmoi in \cite{riehlverity}. We can think of the result in simplicial type theory, as commented in \cite{riehlshulman}, as a directed version of path transport in type theory. Also, there is a dependent version of Yoneda lemma which can be seen as the analogous to path induction for covariant families.
Although we will need only one of its consequences, namely \textit{representability}, there is no harm if we quickly review it in full generality.

Let $C:A\rightarrow \universetyp$ a covariant family and a fix term $a:A$. Then there are functions

\[\evid_a^C:\equiv \lambda\phi.\phi(a,\idx_a):\left(\prod\limits_{x:A}\homtyp_A(a,x)\rightarrow C(x)\right)\rightarrow C(a),\]
\[\yon_a^C:\equiv \lambda u.\lambda x.\lambda f.f_*u: C(a)\rightarrow \left(\prod\limits_{x:A}\homtyp_A(a,x)\rightarrow C(x)\right).\]

\begin{theorem}(Yoneda lemma, \cite[Theorem 9.1]{riehlshulman})
	If $A$ is a Segal type, then for any covariant family $C:A\rightarrow \universetyp$ and $a:A$ the functions $\yon_a^C$ and $\evid_a^C$ are mutual inverses.
\end{theorem}

We have the dependent version of Yoneda:

\begin{theorem}(\cite[Theorem 9.5]{riehlshulman})
	If $A$ is a Segal type, $a:A$ and $C:\prod\limits_{x:A}(\homtyp_a(a,x)\rightarrow\universetyp)$ is covariant then the function
	\[\evid_a^C:\equiv \lambda\phi.\phi(a,\idx_a):\left(\prod\limits_{x:A}\prod_{f:\homtyp_A(a,x)}C(x,f)\right)\rightarrow C(a,\idx_a)\]
	is an equivalence.
\end{theorem}

One particular instance of this equivalence is when $C:\equiv  \homtyp_A(a',-)$, where $A$ is a Segal type and $a':A$, gives the so-called \textbf{Yoneda embedding}. The concept of representability is closely related to this embedding.

\begin{definition}
	A covariant family $C:A\rightarrow\mathcal{U}$ over a Segal type is \textbf{representable} is there exist $a:A$ and a family of equivalences:
	\begin{center}
		$\prod\limits_{x:A}(\homtyp_A(a,x)\simeq C(x))$.
	\end{center}
	
\end{definition}

To finish this section we state the result that for us is of special interest:

\begin{definition}\label{definitioninitiality}
	A term $b:B$ is \textbf{initial} if for any $x:B$ the type $\homtyp_B(b,x)$ is contractible. That is, if the type
	\begin{center}
		$\textup{\textsf{isinitial}}(b):\equiv \prod\limits_{x:B}\textup{\textsf{iscontr}}(\homtyp_B(b,x))$
	\end{center}
	is inhabited.
	Also, a term $b:B$ is \textbf{terminal} if for any $x:B$ the
	type $\homtyp_B(x,b)$ is contractible. That is, if the type
	\begin{center}
		$\textup{\textsf{isterminal}}(b):\equiv \prod\limits_{x:B}\textup{\textsf{iscontr}}(\homtyp_B(x,b))$
	\end{center}
	is inhabited.
      \end{definition}

      In a Segal type initial and terminal elements are unique up to isomorphism.
      
\begin{proposition} \label{initial-terminal:unique}
  Let $A$ be a Segal type. The following is true:
  \begin{enumerate}
  \item If $a:A$ is an initial element, then it is unique up to isomorphism.
  \item If $a:A$ is a terminal element, then it is unique up to isomorphism.
  \end{enumerate}
\end{proposition}
\begin{proof}
Assume that $A$ is a Segal type and let $a,b:A$ be initial elements. The types $\homtyp_A(a,b),\homtyp_A(b,a)$ are contractible, in particular, inhabited say by $f,g$, respectively. We get the terms $f\circ g :\homtyp_A(b,b),g\circ f :\homtyp_A(a,a)$ and $\textup{\textsf{id}}_b:\homtyp_A(b,b), \textup{\textsf{id}}_a:\homtyp_A(a,a)$. Again, these hom types are contractible so $f\circ g=\textup{\textsf{id}}_b$ and $g\circ f=\textup{\textsf{id}}_a$. By \Cref{proposition101} $f$
is an isomorphism between $a$ and $b$. Therefore, initial terms in Segal types are unique up to isomorphism. Similarly, a terminal element is unique up to isomorphism.  
\end{proof}

\begin{proposition}\label{proposition910}
	A covariant type family $C : A\rightarrow \mathcal{U}$ over a Segal type $A$ is representable if
	and only if there exists an initial element $(a, u)$ in the type $\sum\limits_{x:A}C(x)$, in which case
	\begin{center}
		$\textup{\textsf{yon}}_a^C(u):\prod\limits_{x:A}(\homtyp_A(a,x)\rightarrow C(x))$ 
	\end{center}
	defines an equivalence.
\end{proposition} 
\begin{proof}
	This is \cite[Proposition 9.10]{riehlshulman}.
\end{proof}

The majority of the results we have cited are stated in terms of covariant families. However, for contravariant families we have similar results. The proofs for these statements are completely analogous.


\section{Limits and colimits}\label{sectiondefinitions}

There is a well-known path to define limits and colimits in $\infty$-categories, as in \cite{lurie} for quasi-categories or in Segal spaces \cite{yonedarasekh}. This strategy
is first to establish an adequate notion of cones and cocones. The second step is to define initial and terminal objects. Then a \textit{limit} is a ``terminal cone'' and a \textit{colimit} is an ``initial cocone.'' We follow the same approach to define them in sHoTT using the machinery of covariant families.\\

For $A,B$ types and $b:B$ define the constant function
\[
\triangle b:\equiv  \lambda a.b:A\rightarrow B.
\]

\begin{remark}\label{homfuntypcov}
	If $B$ is a Segal type then by \Cref{rmkcorollary56}, $B^A:\equiv A\rightarrow B$ is also a Segal type. Let $f:A\rightarrow B$, from \Cref{proposition813} and the fact that covariance and contravariance are stable under substitution (precomposition) then the type families
	\begin{center}
		\begin{tikzcd}[sep=small]
			\lambda b.\homtyp_{B^A}(f,\triangle b): B \ar[r] & \mathcal{U}
		\end{tikzcd}
	\text{ and }
		\begin{tikzcd}[sep=small]
			\lambda b.\homtyp_{B^A}(\triangle b,f): B \ar[r] & \mathcal{U}
		\end{tikzcd}
	\end{center}
	are covariant and contravariant, respectively.
\end{remark}

\begin{definition} \label{conesdefinition}
	Let $f:A\rightarrow B$ be a function. Define the type of \textbf{cocones} of $f$ as the type
	\[\cocone(f):\equiv \sum\limits_{b:B}\homtyp_{B^A}(f,\triangle b),\] and the type of \textbf{cones} of $f$ as the type
		\[\cone(f):\equiv \sum\limits_{b:B}\homtyp_{B^A}(\triangle b,f).\]
      \end{definition}
      
      \begin{observation}
        From \Cref{proposition65} it follows that
\[
\cocone(f)\simeq \sum\limits_{b:B}\prod\limits_{a:A}\homtyp_B(fa,b)
\]
and
\[
\cone(f)\simeq \sum\limits_{b:B}\prod\limits_{a:A}\homtyp_B(b,fa).
\]
      \end{observation}

We can recognize the type of cones and cocones of a function as direct analogues to cones and cocones in category theory. In \Cref{semanticscomparisonlc} we will verify how the definition above, and the following, are consistent with the semantics.

\begin{definition} \label{colimitdef}
	A \textbf{colimit} for $f:A\rightarrow B$ is an initial term of the type:
	\[
		\cocone(f)\equiv \sum\limits_{x:B}\homtyp_{B^A}(f,\triangle x).
	\]
	A \textbf{limit} for $f:A\rightarrow B$ is an terminal term of the type:
	\[
		\cone(f)\equiv \sum\limits_{x:B}\homtyp_{B^A}(\triangle x,f)
	\]
\end{definition}

\begin{remark} \label{cones:segal}
	Since the total space of a covariant family over a Segal type is a Segal type then from \Cref{homfuntypcov} it follows that the type of cocones of a function
	$f:A \to B$ is a Segal type. Similarly, we also get that the type of cones is a Segal type. We will make use of these two facts where needed without explicit mention.
\end{remark}

 We can show the following:

\begin{corollary}\label{uniqueuptoiso}
	Let $B$ be a Segal type and $f:A\rightarrow B$ is a function. Limits and colimits are unique up to a unique isomorphism if they exist.
\end{corollary}
\begin{proof}
	It follows immediately from \Cref{cones:segal} and \Cref{initial-terminal:unique} above.
\end{proof}

We will obtain in \Cref{uniquelimits} uniqueness property up to equality under the additional assumption that $ B $ is a Rezk type. The following characterization is more flexible in terms of computations. We can think of it as the universal property of limits and colimits.

\begin{proposition} \label{univpropcolim}
	Let $B$ a Segal type and a function $f:A\rightarrow B$. There exists a colimit $(b,\beta)$ for $f$ if and only if
	\begin{center}
		$\prod\limits_{x:B}(\homtyp_B(b,x)\simeq \homtyp_{B^A}(f,\triangle x))$.
	\end{center}
	Similarly, there exists a limit $(a,\alpha)$ for $f$ if and only if
	\begin{center}
		$\prod\limits_{x:B}(\homtyp_B(x,a)\simeq \homtyp_{B^A}(\triangle x,f))$.
	\end{center}
\end{proposition}
\begin{proof}
	The existence of a colimit $(b,\beta)$ for the function $f$ \ie $(b,\beta)$ is an initial term of the type $\sum\limits_{x:B}\homtyp_{B^A}(f,\triangle x)$,
	by \Cref{proposition910} is equivalent to say that the covariant family 
	\begin{center} 
		$\homtyp_{B^A}(f,\triangle -):B\rightarrow \mathcal{U}$                               
	\end{center}
	is representable. The proof for the second part is similar.
\end{proof}

In the situation of \Cref{univpropcolim}, it is common practice to say that ``$b$'' is the colimit and that ``$a$'' is the limit of $f:A\to B$, respectively. This leaves implicit the datum of the cocone and the cone, respectively.

\subsection{Preservation of limits and colimits}

The main goal of this section is to reproduce the classical result that limits are preserved under right adjoints. The same argument will show that left adjoints preserve colimits.
Different presentation of adjunction are studied in \cite{riehlshulman}, namely \textbf{transposing} and \textbf{diagrammatic} adjunctions, and
it is shown that they are equivalent. The first type of adjunction is the analogous definition of adjunction in category theory
in terms of hom sets, whereas the second resembles the triangle identities. For us, it will be enough to know that a \textbf{quasi-transposing adjunction} between types
$A,B$ consists of functions $f:A\rightarrow B$ and $u:B\rightarrow A$ and a family of maps

\[
\phi:\prod\limits_{a:A,b:B}\homtyp_B(fa,b)\rightarrow \homtyp_A(a,ub)
\]
with quasi-inverses
\[
\psi:\prod\limits_{a:A,b:B}\homtyp_A(a,ub)\rightarrow \homtyp_B(fa,b)
\]
and homotopies 
\[\xi:\prod_{a,b,k}\phi_{a,b}(\psi_{a,b}(k))=k,\quad\zeta:\prod_{a,b,l}\psi_{a,b}(\phi_{a,b}(l))=l.\]

In this situation $u$, together with the data above is a \textbf{quasi-transposing right adjoint} for $f$. Using this notation we have:

\begin{theorem}\label{rapl}
	Let $A,B$ be Segal types and functions $g:J\rightarrow B$, $f:A\rightarrow B$, $u:B\rightarrow A$. If $g$ has a limit $(b,\beta)$ and $u$ is a quasi-transposing right adjoint of $f$, then $(u(b),u\beta)$ is a limit for $ug:J\rightarrow A$.
\end{theorem}
\begin{proof}
	Denote by
	\[
	D:\equiv \homtyp_{A^J}(\triangle -,ug):A\rightarrow \mathcal{U} \text{ and } C:\equiv \homtyp_{B^J}(\triangle -,g):B\rightarrow \mathcal{U}.
	\]
	Consider the term $(u(b),u\beta):\cone(ug)$. For any $(y,\alpha):\cone(ug)$ we have:
	\[
	\homtyp_{D(k)}(\alpha,u\beta)\equiv\left\langle \prod\limits_{s:\mathbbm{2}}D(k(s))\middle|_{[\alpha,u\beta]}^{\partial\simplext} \right\rangle
	\]
	and
	\[
	\prod\limits_{s:\mathbbm{2}}D(k(s))\equiv \prod\limits_{s:\mathbbm{2}}\homtyp_{A^J}(\triangle k(s),ug)\simeq \prod\limits_{s:\mathbbm{2}}\prod\limits_{j:J}\homtyp_A(k(s),ug(j)).
	\]
	For any $h:\homtyp_B(f(y),b)$ and $k:\homtyp_A(y,ub)$ we have a map:
	\[
	\homtyp_{D(k)}(\alpha,u\beta)\rightarrow \homtyp_{C(h)}(f\alpha,\beta)
	\]
	defined by $\lambda l.\lambda s.\lambda j.\psi_{k(s),g(j)}(l(s)(j))$. And also
	\[
	\homtyp_{C(h)}(f\alpha,\beta)\rightarrow \homtyp_{D(k)}(\alpha,u\beta)
	\]
	is defined by $\lambda m.\lambda s.\lambda j.\phi_{m(s),g(i)}(m(s)(j))$. These maps compose to the identities because of the coherence given by
	$\xi$ and $\zeta$, thus
	\[
	\homtyp_{D(k)}(\alpha,u\beta)\simeq \homtyp_{C(h)}(f\alpha,\beta).
	\]	
	Hence:
	\begin{align*}
		\homtyp_{\cone(ug)}((y,\alpha),(u(b),u\beta)) & \simeq \sum\limits_{k: \homtyp_A(y,ub)}\homtyp_{D(k)}(\alpha,u\beta) \\
		& \simeq \sum\limits_{h: \homtyp_B(fy,b)}\homtyp_{C(h)}(f\alpha,\beta) \\
		& \simeq \homtyp_{\cone(g)}((fy,f\alpha),(b,\beta))  
	\end{align*}
	The third equivalence follows from the adjunction and the equivalence above. This last type is contractible since $(b,\beta)$ is a limit for $g$.
	In conclusion, $(u(b),u\beta)$ is a limit for $ug$.
\end{proof}

We also have:

\begin{corollary}
	If $A,B$ are Segal types and $f:A\rightarrow B$ has left quasi-transposing adjunct then $f$ preserves colimits of functions $g:J\rightarrow A$.
\end{corollary}
\begin{proof}
	Similar to the previous theorem.
\end{proof}

\section{(Co)Limits in Rezk types}\label{sectionrezk}

In this section, we restrict our attention to colimits in a Rezk type, which allow to improve some of our results.

In what follows, for $C:A\rightarrow \mathcal{U}$ a covariant (contravariant) family we denote by \[\tilde{C} :\equiv \sum\limits_{x:A}C(x)\]
to the dependent sum over $A$.

\begin{lemma}\label{isoequivdepsum}
	Assume that $C:A\rightarrow \mathcal{U}$ is a covariant family over a Segal type $A$, let $(a,u),(b,v):\sum\limits_{x:A}C(x)$. Then
	\[
	((a,u)\cong (b,v))\simeq \sum\limits_{f:a\cong b}f_*u=v.
	\]
\end{lemma}
\begin{proof}
	First, we have the equivalence
	\begin{align}\label{equivtot}
		\homtyp_{\tilde{C}}((a,u),(b,v))\simeq \sum\limits_{f:\homtyp_A(a,b)}\homtyp_{C(f)}(u,v).
	\end{align}
	Recall that this follows from \Cref{theorem43}
	
	\[
	\left\langle \prod_{t:I\: | \:\psi}(\sum_{x:X(t)}Y(t,x))\middle|_{\lambda t.(a(t),b(t))}^{\phi}\right\rangle \simeq \sum_{f:\left\langle\prod_{t:I\: | \:\psi}X(t)\middle|_a^{\phi}\right\rangle}\left\langle\prod_{t:I\: | \:\psi}X\middle|_f^{\psi}\right\rangle
	\]
	where the proof is established by tracing the composition of the maps given from left to right by
	\[h\mapsto (\lambda t.\pi_1(h(t)),\lambda t.\pi_2(h(t))),\]
	and from right to left by
	\[(f,g)\mapsto \lambda t.(f(t),g(t)).\]
	This applied to (\ref{equivtot}) yields in the first coordinate the projection
	\[
	\pi_1:\sum\limits_{x:A}C(x)\rightarrow A.
	\]
	Since $C:A\rightarrow \mathcal{U}$ is a covariant family over the Segal type $A$  then $\tilde{C}$ is a Segal type \cite[Theorem 8.8]{riehlshulman}. By \Cref{proposition61} the projection respects identities and composition. In particular, if $\bar{f}:(a,u)\cong (b,v)$ then $\pi_1(\bar{f}):a\cong b$.
	Furthermore, covariance of $C$ implies that $(v,\pi_2(\bar{f}))=(\pi_1(\bar{f})_*u,\trans_{\pi_1(\bar{f}),u})$ hence $\pi_1(\bar{f})_*u=v$.
	Whence we can define
	\[
	\phi:((a,u)\cong (b,v))\rightarrow \sum\limits_{f:a\cong b}f_*u=v
	\]
	to be $\phi:\equiv \lambda \bar{f}.(\pi_1(\bar{f}),p)$ where $p:\pi_1(\bar{f})_*u=v$. \\
	In order to construct the function in the other direction we observe the following: if $f:a\cong b$ then there exist $g:\homtyp_A(b,a)$ with 
	$f\circ g=\textup{\textsf{id}}_b$ and $g\circ f=\textup{\textsf{id}}_a$. The covariance of $C$ implies that the types
	\[
	\sum\limits_{w: C(b)}\homtyp_{C(f)}(u,w) \text{ and } \sum\limits_{w: C(a)}\homtyp_{C(g)}(v,w)
	\]
	are contractible with center of contraction $(f_*u,\textup{\textsf{trans}}_{f,u})$ and $(g_*v,\textup{\textsf{trans}}_{g,v})$, respectively. Note also that 
	\[
	u=(\textup{\textsf{id}}_a)_*u=(g\circ f)_*u=g_*(f_*u)=g_*v.
	\]
	We get elements
	\[
	(f,\textup{\textsf{trans}}_{f,u}):\sum\limits_{h:\homtyp_A(a,b)}\homtyp_{C(h)}(u,v)
	\]
	and
	\[
	(g,\textup{\textsf{trans}}_{g,v}):\sum\limits_{h:\homtyp_B(b,a)}\homtyp_{C(h)}(v,u).
	\]
	Using the equivalence
	\begin{align}\label{homtotal}
		\left(\sum\limits_{k: \homtyp_A(x,y)}\homtyp_{C(k)}(w,z)\right)\simeq \homtyp_{\tilde{C}}((x,w),(y,z))
	\end{align}
	we get $\bar{f}:\homtyp_{\tilde{C}}((a,u),(b,w))$ and $\bar{g}:\homtyp_{\tilde{C}}((b,v),(a,u))$. Since $\tilde{C}$ is a Segal type then $\bar{g}\circ\bar{f}: \homtyp_{\tilde{C}}((a,u),(a,u))$. We have the composition in $A$ and the dependent composition, so we obtain the 2-simplex
	\begin{center}
		$\textup{\textsf{comp}}_{g,f}:\homtyp_A^2\left(\begin{tikzcd}[sep=tiny]
			& b \ar[dr,dash,"g"] & \\
			a \ar[ur,dash,"f"] \ar[rr,dash,"\textup{\textsf{id}}_a"'] & & a \\
			 & \text{} &
		\end{tikzcd}\right)$
	\end{center}
	and
	\begin{center}
		$\textup{\textsf{comp}}_{\textup{\textsf{trans}}_{g,v},\textup{\textsf{trans}}_{f,u}}:\homtyp_{C(t)}^2\left(\begin{tikzcd}[sep=tiny]
			& v \ar[dr,dash] & \\
			u \ar[ur,dash] \ar[rr,dash,"\textup{\textsf{id}}_u"'] & & u \\
			& \text{} & 
		\end{tikzcd}\right).$
	\end{center}
	We get a term $\big(\overline{g\circ f},(\textup{\textsf{comp}}_{g,f},\textup{\textsf{comp}}_{\textup{\textsf{trans}}_{g,v},\textup{\textsf{trans}}_{f,u}})\big)$ in
	\begin{center}
		$\sum\limits_{h:\homtyp_{\tilde{C}}((a,u),(a,u))}\homtyp_{\tilde{C}}^2\left( \begin{tikzcd}[sep=tiny]
			& (b,v) \ar[dr,dash,"\bar{g}"] & \\
			(a,u) \ar[ur,dash,"\bar{f}"] \ar[rr,dash,"h"'] & & (a,u) \\
			& \text{} & 
		\end{tikzcd}\right)$.
	\end{center}
	Since this type is contractible, in particular we have $\bar{g}\circ\bar{f}=\overline{g\circ f}$. Under the equivalence (\ref{homtotal}), $\textup{\textsf{id}}_a$ 
	is mapped to $\textup{\textsf{id}}_{(a,u)}$, so \[\bar{g}\circ\bar{f}=\overline{g\circ f}=\overline{\textup{\textsf{id}}}_a=\textup{\textsf{id}}_{(a,u)}.\] Similarly, $\bar{f}\circ\bar{g}=\textup{\textsf{id}}_{(b,v)}$.
	Therefore we define
	\[
	\psi:\sum\limits_{f:a\cong b}f_*u=v\rightarrow ((a,u)\cong (b,v))
	\]	
	as $\psi:\equiv \lambda(f,p).\bar{f}$. The functions $\phi$ and $\psi$ give us the required equivalence.
	
\end{proof}

A related result is \cite[Theorem 2.7.2]{hottbook}, instead of identity types we now have isomorphism types. Thus, an isomorphisms in $\Sigma$-types is determined by an isomorphism in the base together with a dependent isomorphism in the total type, lying over it. We formulate the result, but give no proof, corresponding to contravariant families.

\begin{lemma}\label{isoequivdepsum:contravariant}
	Assume that $C:A\rightarrow \mathcal{U}$ is a contravariant family over a Segal type $A$, let $(a,u),(b,v):\sum\limits_{x:A}C(x)$. Then
	\[
	((a,u)\cong (b,v))\simeq \sum\limits_{f:a\cong b}u=f^*v.
	\]
\end{lemma}

\begin{theorem}\label{totalisrezk}
	If $A$ is a Rezk type and $C:A\rightarrow\mathcal{U}$ is a covariant (contravariant) family over $A$ then $\sum\limits_{x:A}C(x)$ is also
	a Rezk type.
\end{theorem}
\begin{proof}
	Let $(a,u),(b,v):\sum\limits_{x:A}C(x)$. We have
	\begin{align*}
		\big((a,u)=(b,v)\big)\simeq & \sum\limits_{f:a=b}\textup{\textsf{transport}}^C(f,u)=v \\
		\simeq & \sum\limits_{\textup{\textsf{idtoiso}}(f):a\cong b}\textup{\textsf{idtoiso}}(f)_*u=v \\
		\simeq & (a,u)\cong(b,v),
	\end{align*}
	the middle equivalence follows from $A$ being Rezk type and from \Cref{lemma107} where we have the equality $\textup{\textsf{idtoiso}}(e)_* u = \textup{\textsf{transport}}^C (e, u)$.
	The function
	\[
	\textup{\textsf{idtoiso}}:(a,u)=(b,v)\rightarrow (a,u)\cong(b,v)
	\]
	is exactly the equivalence above. Indeed, we can use path induction to see this. If $(a,u)=(a,u)$ then $\textup{\textsf{refl}}_{(a,u)}$ is mapped to $(\textup{\textsf{refl}}_a,\textup{\textsf{refl}}_u)$. This goes to $(\textup{\textsf{id}}_a,\textup{\textsf{refl}}_u)$ 
	which finally corresponds to $\textup{\textsf{id}}_{(a,u)}:(a,u)\cong (a,u)$. Therefore, $\sum\limits_{x:A}C(x)$ is Rezk.
      \end{proof}

      \begin{remark}
        \Cref{totalisrezk} can also be deduced from Proposition 4.2.6 and Corollary 6.1.4 of \cite{buchwein}. Their proof makes use of their theory of \textit{(co)cartesian fibrations} in sHoTT that they develop throughout their work, while ours is more elementary and simply uses the characterization of the type of isomorphisms between elements in a $\sigma$-type we provided in \Cref{isoequivdepsum}.
      \end{remark}

\begin{definition}
	Let $B$ is a Segal type, $f:A\rightarrow B$ a function. We define the type:
	
	\[
	\textup{\textsf{hasColimit}}(f):\equiv \sum\limits_{w:\textup{\textsf{cocone}}(f)}\textup{\textsf{isinitial}}(w)
	\]
	
	and also
	
	\[
	\textup{\textsf{hasLimit}}(f):\equiv \sum\limits_{w:\cone(f)}\textup{\textsf{isterminal}}(w).
	\]
	
\end{definition}

In \Cref{uniqueuptoiso} we showed that limits and colimits are unique up to isomorphism. Now, we can show further:

\begin{corollary} \label{uniquelimits}
	If $B$ is a Rezk type and we have a function $f:A\rightarrow B$. The types $\textup{\textsf{hasColimit}}(f)$ and $\textup{\textsf{hasLimit}}(f)$ are propositions.
\end{corollary}
\begin{proof}
	Let $(a,\alpha), (b,\beta)$ be limits for $f$. By \Cref{uniqueuptoiso} we have that $(a,\alpha)\cong (b,\beta)$. \Cref{totalisrezk} applied to the covariant family
	
	\[
	\homtyp_{B^A}(f,\triangle -):B\rightarrow \mathcal{U}
	\]
	implies that $\cone(f)$ is also a Rezk type. Therefore:
	\[
	\big((a,\alpha)= (b,\beta)\big)\simeq \big((a,\alpha)\cong (b,\beta)\big)
	\]
	hence $\textup{\textsf{hasLimit}}(f)$ is a proposition. A similar proof shows that $\textup{\textsf{hasColimit}}(f)$ is also a proposition.
\end{proof}

\section{Limit of spaces as dependent product}\label{sectioncomputation}

Let $\{G_i\}_{i\in I}$ be a family of $\infty$-groupoids indexed by a set $I$. Denote by $G$ to the obvious diagram $I \to \infty\text{-}\textbf{Gpd}$ then we have that 
\[
\lim_I\, G_i=\prod_{i\in I}G_i.
\]
We can think for simplicity that the calculation above takes place in bisimplicial sets \ie this is a family of complete Segal spaces which are homotopically constant. Or could also be happening in simplicial sets.

Given the previous result, it seems natural to expect that a similar result is true in sHoTT. However, such computation is not straightforward: The obstruction is that in our type theory we do not have an $\infty$-category of $\infty$-groupoids in which we can take this limit. The way we go around this issue is by using the notion of \textit{univalent covariant families} originally due to Cavallo-Riehl-Sattler \cite{riehlcavallosattler} to sufficiently axiomatize the properties of the $\infty$-category of spaces.

Let $B$ a Segal type. For any covariant family $E:B\rightarrow \mathcal{U}$ and $a,b:B$, using the notation of \Cref{covardefi} we obtain a function
\[
\textup{\textsf{arrtofun}}:\homtyp_B(a,b)\rightarrow \big(E(a)\rightarrow E(b)\big)
\]
where for any $f: \homtyp_B(a,b)$ we set
\[
\textup{\textsf{arrtofun}}(f):\equiv f_*\equiv \lambda u.f_*u.
\]

\begin{definition}\label{directedunivalence}
	
	Given $E:B\rightarrow\mathcal{U}$ a covariant family over a Segal type $B$. We say that $E$ is \textbf{univalent} if for all $a,b:B$ 
	the map $\textup{\textsf{arrtofun}}$ is an equivalence. Denote its inverse by
	
	\[
	\dua:\big(E(a)\rightarrow E(b)\big)\rightarrow \homtyp_B(a,b).
	\]
\end{definition}

We can understand the type $B$ as satisfying a ``directed univalent axiom.'' The prime example of a \textit{univalent fibration} in the semantics should be the $\infty$-category of $\infty$-groupoids. In \cite[Remark A.27]{riehlshulman} it is shown that covariant families correspond to left fibrations. Cavallo, Riehl, and Sattler announced in \cite{riehlcavallosattler} (unpublished) the existence of a fibrant universe (bisimplicial set) of left fibrations. This implies the consistency of \Cref{directedunivalence}, see also \Cref{directedunivalence:remark}.

The result in \cite[Proposition 8.18]{riehlshulman} shows that if $ E:B \to \universetyp $ is a covariant type family over a Segal type $B$, then for any $b:B$ the type $E(b)$ is discrete. Furthermore, if $E$ is univalent we could think of $B$ as a ``Segal type that has terms discrete types.'' We will make this idea precise in \Cref{issmallisprop}. In this part of the paper, we make use freely of the extensionality axiom, and the notation introduced in \Cref{extensionality}. In particular, we use $f\sim g$ introduced there.

\begin{lemma}\label{inducesiso}
	Let $B$ a Segal type and $E:B\rightarrow \mathcal{U}$ a univalent covariant family and $a,b:B$. If $\delta:E(a)\rightarrow E(b)$ is an equivalence then $\dua(\delta)$ is an isomorphism.
\end{lemma}
\begin{proof}
	By \Cref{proposition816} for any $x:B$ and $u:E(x)$ we have $(\idx_{x})_*(u)=u$. Therefore, $\arrtofun(\idx_x)=\idx_{E(x)}$. This also implies that	
	\[
	\dua(\idx_{E(x)})=\dua(\arrtofun(\idx_x))=\idx_x.
	\]
	Treat $\delta:E(a)\rightarrow E(b)$ as a bi-invertible function so there is $\gamma:E(b)\rightarrow E(a)$ such that $\delta\gamma\sim \idx_{E(b)}$ and $\gamma\delta\sim \idx_{E(a)}$.
	From the extensionality axiom we can further assume $\delta\gamma=\idx_{E(b)}$ and $\gamma\delta=\idx_{E(a)}$. Hence, $\dua(\delta\gamma)=\idx_b$ and $\dua(\gamma\delta)=\idx_a$. 
	Observe that again from \Cref{proposition816}, for any $f : \homtyp_B(x, y)$, $g : \homtyp_B(y, z)$, and $u : E(x)$, 
	we have $g_* (f_* u) = (g\circ f )_* u$, and as a consequence
	\[
	\arrtofun(\dua(\delta)\circ\dua(\gamma))=\arrtofun(\dua(\delta))\arrtofun(\dua(\gamma)))=\delta\gamma=\idx_{E(b)}
	\]
	it follows that $\dua(\delta)\circ\dua(\gamma)=\dua(\delta\gamma)=\idx_b$. Similarly	
	$$
	\dua(\gamma)\circ\dua(\delta)=\dua(\gamma\delta)=\idx_a.$$
        Which concludes the proof of the statement.
\end{proof}

\begin{definition} \label{issmall:def}
	Given a type $A$ and a covariant family $E:B\rightarrow \mathcal{U}$ over a type $B$, we define the type
	\[
	\issmall^E_B(A):\equiv \sum\limits_{b:B}\big(E(b)\simeq A\big).
	\]
	We say that $A$ is $B$-small if $\issmall^E_B(A)$ is inhabited.
\end{definition}

A related notion appears in \cite[Definition 17.1.3]{hottRijke} where they define what it means for a type $A$ to be small with respect to a univalent universe $\universetyp$; A type $A$ is said to be $\universetyp$-\textbf{small} if there is type $X:\universetyp$ and an equivalence between $A$ and $X$. The same definition appear in \cite[Section 2.19]{Symmetry} under the name \textit{essentially $\universetyp$-small}. From our definition, since we require a covariant family, we can express whether a discrete type $A:\universetyp$ is $\universetyp$-small. This is because if $A:\universetyp$ and $E:\pointtyp \to \universetyp$ is a covariant family with value $A$, then $A$ must be discrete. And it is obvious that $A$ is $\universetyp$-small. A more general and concrete example of this can be seen in the semantics. Given a Grothendieck universe $U$, one could build a sub-universe $U'$ (an $\infty$-category) which consists of $\kappa$-small $\infty$-groupoids, where $\kappa$ is a regular cardinal. By construction, elements of $U'$ are exactly the $U'$-small elements of $U$. Of course, this would be one way to justify the consistency of \Cref{directedunivalence}.

The idea in \Cref{issmall:def} is that $B$ has an element representing the type $A$ in question. We can think of $A$ as ``a type in $B$'', hence the name $B$-small. To make sense of this, we need to note some simple facts. The first lemma simply computes $\transport^D$ for the type family $D:B \to \universetyp$ as defined below.

\begin{lemma} \label{transport-along-equivalence:lem}
  Let $E:B\rightarrow \mathcal{U}$ be a univalent covariant family over a Rezk type $B$. For a type $A$, define $D:B\rightarrow\mathcal{U}$ as $D(b):\equiv E(b)\simeq A$.
Then for any $b,b':B$, $p:b=b'$ and $\sigma:E(b)\simeq A$
\[
\transport^D(p,\sigma)=\sigma(\arrtofun(\idtoiso(p^{-1}))).
\]
\end{lemma}
\begin{proof}
  
By path induction we can assume that $p\equiv \refl_b:b=b$. In this case 
\[
\transport^D(\refl_b,\sigma)\equiv \idx_{D(b)}(\sigma)=\sigma.
\]
Whereas on the right-hand side we have
\begin{align*}
	\sigma(\arrtofun(\idtoiso(p^{-1}))) &\equiv \sigma(\arrtofun(\idtoiso(\refl_b))) \\
	&\equiv \sigma(\arrtofun(\idx_b))\\ 
	&\equiv\sigma\idx_{E(b)} \\
	&=\sigma
\end{align*}
hence we can simply use $\refl_{\sigma}:\sigma=\sigma$.
Additionally, path induction also shows that for any $p:b=b'$
\[
\arrtofun(\idtoiso(p^{-1}))=\arrtofun(\idtoiso(p))^{-1}
\]
since if $p\equiv \refl_b$ then both sides of the equality are $\idx_{E(b)}$. Here, $\arrtofun(\idtoiso(p))^{-1}$ abusively denotes an inverse for $\arrtofun(\idtoiso(p))$.
\end{proof}
This has the following consequence:

\begin{corollary}\label{issmallisprop}
	Let $E:B\rightarrow \mathcal{U}$ be a univalent covariant family over a Rezk type $B$. Then for any type $A$, $\issmall^E_B(A)$ is a proposition.
\end{corollary}
\begin{proof}
	Let $(b,\sigma),(b',\sigma'):\sum\limits_{b:B}\big(E(b)\simeq A\big)$. We have
	\[
	\big((b,\sigma)=(b',\sigma')\big)\simeq \sum\limits_{p:b=b'}\sigma=^p\sigma'
	\]
	therefore it is enough to give a term of the right-hand side. We realize $\sigma,\sigma'$ as a bi-invertible functions, let $\delta:A\rightarrow E(b)$, $\delta':A\rightarrow E(b')$ be their inverses, respectively. We get an equivalence $\delta'\sigma:E(b)\simeq E(b')$, by \Cref{inducesiso} $\dua(\delta'\sigma):b\cong b'$. Since $B$ is a Rezk type, $\rezk(\dua(\delta'\sigma)):b=b'$. From the \Cref{transport-along-equivalence:lem} above and the fact that $((\arrtofun)(\idtoiso))((\rezk)(\dua))=\idx_{E(b')^{E(b)}}$  we obtain:
	\begin{align*}
		\transport^D(\rezk(\dua(\delta'\sigma)),\sigma) &=\sigma(\arrtofun(\idtoiso(\rezk(\dua(\delta'\sigma))^{-1})) \\
		&=\sigma(\arrtofun(\idtoiso(\rezk(\dua(\delta'\sigma))))^{-1}) \\
		&=\sigma(\delta'\sigma)^{-1} \\
		&=\sigma(\delta\sigma')\\
		&=\idx_{A}\sigma' \\
		&=\sigma'
	\end{align*}
	where again $(\delta'\sigma)^{-1}$ denotes an inverse for $\delta'\sigma$, but since $\delta\sigma'$ is also an inverse these are equal.	
\end{proof}

\begin{remark}
	In the situation of \Cref{issmallisprop}, if a type $A$  is $B$-small then $\issmall^E_B(A)$ is contractible. This entails the existence of a center of contraction $(b_0,\sigma_0)$ where $b_0$ is a term of type $B$ and $\sigma_0$ is an equivalence between $E(b_0)$ and $A$. Hence, $A$ is uniquely related to a term in $B$.
\end{remark}

Therefore, given a directed univalent family of types $E:B\rightarrow \mathcal{U}$, one can think of $B$ as being the $\infty$-category of ``small spaces'', or at least a full subcategory of it. However, note that we do not assume that $B$ is closed under limits. The next proposition shows that as long as a dependent product of small types is itself small, it corresponds to a limit in $B$.

\begin{proposition} \label{condition1}
	Let $B$ a Rezk type, a function $f:D\rightarrow B$ and assume that $E:B\rightarrow \mathcal{U}$ is a univalent covariant family such that $\prod\limits_{d:D}E(f(d))$ is $B$-small. If $(b_0,\sigma_0)$ is the center of contraction of $\issmall^E_B(\prod\limits_{d:D}E(f(d)))$ then $b_0$ is the limit for the function $f$.
\end{proposition}
\begin{proof}
	Indeed, using \Cref{univpropcolim} it is enough
	to show that for all $b:B$
	
	\[
	\homtyp_B(b,b_0)\simeq \homtyp_{B^D}(\triangle b,f).
	\]
	For the right-hand side, we have:	
	\[
	\homtyp_{B^D}(\triangle b,f)\simeq\left( \prod\limits_{d:D}\homtyp_{B}(b,f(d))\right) \simeq \left(\prod\limits_{d:D}E(b)\rightarrow E(f(d)) \right)
	\]
	The last equivalence follows because $E$ is univalent. For the same reason together with the assumption $\sigma_0:E(b_0)\simeq \prod\limits_{d:D}E(f(d))$
	the left-hand side is equivalent to	
	\[
	\big(E(b)\rightarrow E(b_0)\big) \simeq \left(E(b)\rightarrow\prod\limits_{d:D}E(f(d))\right).
	\]
	And certainly we have the equivalence
	\[
	\left(E(b)\rightarrow\prod\limits_{d:D}E(f(d))\right)\simeq \left(\prod\limits_{d:D}E(b)\rightarrow E(f(d))\right).
	\]
	
\end{proof}

For the rest of the section we want to show that the condition over the dependent product of the fiber being $B$-small in \Cref{condition1} is necessary for the existence of the limit.

\begin{observation} \label{contraction-center:observation}
Suppose that $D:B\rightarrow \universetyp$ is a constant covariant family at $D:\universetyp$ over a Segal type $B$. Take any $f:\homtyp_B(x,y)$ and $w:D$
then the type
\[\sum\limits_{v:D}\homtyp_{D(f)}(w,v)\]
has center of contraction $(f_*w,\trans_{f,w})$. The point $(w,\lambda (t:\mathbbm{2}).w)$ belongs to the aforementioned type, which then must be equal to the center of contraction. In particular this implies that $f_*w=w$.  
\end{observation}

This observation has the following consequence:

\begin{lemma}\label{generilized}
	Let $B:\universetyp$ a Segal type and $E:B\rightarrow \universetyp$ a covariant univalent family over $B$. If there is $u:E(b_0)$ for some $b_0:B$
	then for all $b_1:B$ \[ E(b_1)\simeq \prod\limits_{b:B}\homtyp_B(b,b_1).\]
\end{lemma}
\begin{proof}
	The covariant family $E:B\rightarrow\universetyp$ is univalent, so instead we may prove that for all $b_1:B$	
	\[ E(b_1)\simeq \left(\prod\limits_{b:B}E(b)\rightarrow E(b_1)\right).\]
	Let $b_1:B$ and define
	\[ F:E(b_1)\rightarrow \left(\prod\limits_{b:B}E(b)\rightarrow E(b_1)\right),\]
	where $F(v):\equiv F_v$ is the dependent function $\lambda(b:B).\lambda(w:E(b)).v$. The function in the other direction
	\[ G:\left(\prod\limits_{b:B}E(b)\rightarrow E(b_1)\right)\rightarrow E(b_1) \]
	for each $\phi:\left(\prod\limits_{b:B}E(b)\rightarrow E(b_1)\right)$ is defined by $G(\phi):\equiv \phi_{b_0}(u)$. One composition immediately is the identity
	\[ G(F(v))\equiv G(F_v)\equiv (F_v)_{b_0}(u)\equiv v.\]
	For any $\phi:\left(\prod\limits_{b:B}E(b)\rightarrow E(b_1)\right)$ we have $F(G(\phi))\equiv F(\phi_{b_0}(u))\equiv F_{\phi_{b_0}(u)}$ and we need to show that this dependent function is equal to $\phi$. Consider the family
	\[D:B\rightarrow\universetyp\]
	$D(b):\equiv E(b_1)$, which is covariant. For any $b:B$ define $f:E(b)\rightarrow E(b_0)$ as $f(w):\equiv u$. Since the family $E$ is univalent, 
	$\dua(f):\homtyp_B(b,b_0)$ and
        \begin{align} \label{equation:arrtofun}
           \dua(f)_*\equiv \arrtofun(\dua(f))=f
        \end{align}
        where the definitional equality follows by definition of $\arrtofun$.
        
	By naturality, as in \Cref{proposition817}, applied to the function
	$\phi:\prod\limits_{b:B}E(b)\rightarrow D(b)$ and the arrow $\dua(f):\homtyp_B(b,b_0)$, we have the following commutative square
\[\begin{tikzcd}
	{E(b)} && {E(b_0)} \\
	{D(b)} && {D(b_0)}
	\arrow["{\dua(f)_*}", from=1-1, to=1-3]
	\arrow["{\phi_b}"', from=1-1, to=2-1]
	\arrow["{\phi_{b_o}}", from=1-3, to=2-3]
	\arrow["{\dua(f)_*}"', from=2-1, to=2-3]
\end{tikzcd}\]
where the horizontal functions are given by the specified type families. We now evaluate at any $w:E(b)$. Since $D$ is constant over $B$, from \Cref{contraction-center:observation} we get the first equality below, the second follows by naturality and the third equality comes from (\ref{equation:arrtofun}),
	\begin{align*}
		\phi_b(w) &= \dua(f)_*(\phi_b(w)) \\
		&= \phi_{b_0}(\dua(f)_*(w)) \\
		&= \phi_{b_0}(f(w)) \\
		&= \phi_{b_0}(u) \\
		&\equiv (F_{\phi_{b_0}(u)})_b(w),
	\end{align*}
        while the second to last and last equality follow by definition of $f$ and $F$, respectively. This establishes the equality $\phi=F_{\phi_{b_0}(u)}$. Therefore, $FG$ is the identity on $\prod\limits_{b:B}(E(b)\rightarrow E(b_1))$, hence the equivalence.
\end{proof}

Finally, we have:

\begin{proposition} \label{limit-characterization}
	Let $B$ a Rezk type and a function $f:D\rightarrow B$ with limit $(b_1,\alpha)$. Assume that $E:B\rightarrow \mathcal{U}$ is a univalent covariant family such that for some $b_0:B$ there is $u:E(b_0)$. Then $E(b_1)\simeq \prod\limits_{d:D}E(f(d))$.
\end{proposition}

This just means that as long as the family $E:B\rightarrow \mathcal{U}$ has at least one fiber that is inhabited, then the sufficient condition of \Cref{condition1} is also necessary for the existence of a limit.

\begin{proof}
	From \Cref{generilized} we have
	\[E(b_1)\simeq \prod\limits_{b:B}\homtyp_B(b,b_1),\]	
	since $b_1$ is the limit:
	\[\homtyp_B(b,b_1)\simeq \prod\limits_{d:D}\homtyp_B(b,f(d)).\]
	Combining this equivalences and using \Cref{generilized} again give us:
	\begin{align*}
		E(b_1) & \simeq \prod\limits_{b:B}\prod\limits_{d:D}\homtyp_B(b,f(d)) \\
		& \simeq \prod\limits_{d:D}\prod\limits_{b:B}\homtyp_B(b,f(d)) \\
		& \simeq \prod\limits_{d:D} E(f(d)). \\
	\end{align*}
	
      \end{proof}

      \begin{remark} \label{directedunivalence:remark}
  In a recent paper \cite{gratzer2024directed}, the authors construct the $(\infty,1)$-category of $\infty$-groupoids. They achieve this by enhancing sHoTT with modalities. The type $\mathcal{S}$ they construct satisfies exactly \Cref{directedunivalence}. Our results in this section simply say that a diagram with values in $\mathcal{S}$ has a limit in $\mathcal{S}$. This limit is given by the dependent product. We could rewrite our results by replacing our type $B$ by $\mathcal S$, and each type $E(b)$ for $b:B$ by an element of $\mathcal S$. However, we keep our presentation as it is since we think it showcases a ``resizing'' technique by means of \Cref{issmall:def} that could be useful in other contexts.
\end{remark}


\section{The bisimplicial sets semantics of sHoTT} \label{semantics}

In this this section, we verify that our synthetic definitions of limit an colimit are semantically correct. We use the fact the semantics of sHoTT is the category of bisimplicial sets $ \bisim. $ 

In \cite[Appendix A]{riehlshulman} it is shown that sHoTT can be interpreted in the category $ \bisim $. Let us first recall the following result from \cite{shulmanreedy}:	

\begin{theorem} \label{modelshott}
	For any elegant Reedy category $\catC$, the Reedy model structure on $\sset^{\catC^{op}}$ supports a model of intentional type theory with dependent sums and products, identity types and as many univalent universes as there are inaccessible cardinals greater than $|\catC|$, where $|\catC|$ is the cardinal of the set of objects of the category $\catC$.
\end{theorem}

Let us first review how to interpret some usual types. The one point type is the easiest, it agrees with the terminal object in our model. According to \cite{shulmanreedy} we can take a category $\catM$ which is locally cartesian closed and has a model structure that is simplicial, right proper and the cofibrations are the monomorphisms, this is the situation of \Cref{modelshott}. If we have a map $f:A\to B$ then the pullback functor $f^*:\catM/B \to \catM/A$ has both right and left adjoint which are denoted $\prod_f$ and $\sum_f$. Since fibrations are closed under compositions, the functor $\sum_f$ preserves fibrations if $f$ is itself a fibration. The functor $f^*$ preserve cofibrations since these are exactly the monomorphisms, thus $\prod_f$ preserves acyclic fibrations. Furthermore, one can also see that in combination with right properness $\prod_f$ preserves fibrations, see \cite[Theorem 26]{arndt2011homotopy}. This is how dependent sums and products are interpreted.

The Reedy model structure on bisimplicial sets gives rise to a \textbf{comprehension category with shapes} in which the judgment $\Gamma\vdash A\,\type$ is directly interpreted as a fibration $\Gamma \to A$. Here we are overlooking the fact that our types also may depend on topes and shapes. To see more details and the correct definitions we refer the reader to our cited main reference \cite{riehlshulman}. The technical result is:

\begin{theorem}(\cite[Theorem A.18]{riehlshulman})\label{sttmodel} 
	The comprehension category constructed from any model category with $\mathfrak{T}$-\textbf{shapes} has \textbf{pseudo-stable coherent tope logic} with type eliminations for tope disjunction and also \textbf{pseudo-stable extension types} satisfying relative function extensionality.
\end{theorem}

\cite{riehlshulman} use the result on $\mathfrak T$ the coherent theory of the strict interval of sHoTT presented in \Cref{extensiontypes} or \cite[Section 3.1]{riehlshulman}.
Furthermore, Weinberger shows in \cite{weinberger2022strict} that the semantical interpretation extension types is strictly stable under pullbacks, and not only pseudo stable as in the theorem above. Moreover, \cite[A.3]{riehlshulman} proves:
\begin{theorem} \label{eqsegaltypesegalspace}
	 Segal types correspond to Segal spaces in $ \bisim $ and Rezk types to Rezk spaces (a.k.a. complete Segal spaces).
\end{theorem}

\begin{remark} \label{discreteconstant}
   In \Cref{discretetypes} we defined discrete types. However, the terminology might be confusing. According to \cite[Proposition 10.10]{riehlshulman} a type $ A $ is discrete if and only if $ A $ is Rezk and all its arrows are isomorphisms. On the other hand, \cite[Corollary 6.6]{rezk} gives a similar characterization for complete Segal spaces, such simplicial spaces are \textit{constant}. A simplicial space $ X $ is \textbf{constant} if for any map $ [n] \to [m] $, the induced map $ X_m \to X_n $  is an equivalence of spaces. In particular, for all $ n $ we have $ X_n=X_0 $. Hence, discrete types correspond to constant simplicial spaces. As a consequence, we interpret discrete types into the bisimplicial sets model as constant simplicial spaces.
\end{remark}

\begin{remark} \label{inftytoposinterpretation}
  The interpretation of sHoTT can be done in greater generality. This is essentially a consequence of the main result in \cite{shulman2019all}, where it is shown that any Grothendieck $(\infty,1)$-topos gives a model for Homotopy Type Theory. If $\catC$ is a \textit{type-theoretic model topos} then one can consider simplicial objects internal to $\catE$. The resulting category of internal presheaves gives us again a model for sHoTT. For more details we refer to \cite[Section 2]{weinberger2022strict}.
\end{remark}

Now we have the semantical overview of sHoTT. This gives us the opportunity to say a few words on the necessity of contrasting the definitions we make in the type theory with the existing ones in the bisimplicial sets model. 

\begin{remark} \label{typeimpliessemantics}
  As we mentioned above, sHoTT can be interpreted the category of bisimplicial sets. However, note that a type theoretic statement translates internally to the bisimplicial sets model. A concrete example of this phenomena is the characterization of Segal types. In general, each type is interpreted as a Reedy fibrant bisimplicial set. On the one hand, the equivalence $ A^{\simplextt} \to A^{\horn} $ from \Cref{rmkcorollary56} is interpreted as a trivial fibration in the Reedy model structure on $ \bisim $. On the other hand, Segal spaces are defined by the Segal condition: A bisimplicial set $A$ is a Segal space if for all $n\geq 2$, the map $A_n \to A_1 \times_{A_0} \cdots \times_{A_0} A_1$ is a trivial fibration in the Fallen-Kan model structure on $ \sset $. \cite[A.3]{riehlshulman} show that both conditions are equivalent. 
\end{remark}

Informally speaking, an internal statement in the model could be stronger under interpretation in the intended semantics. The results we present in the following sections show that this is not the case \ie the internal and external notions coincide.

\subsection{Limits and colimits}\label{semanticscomparisonlc}

We verify the consistency of our definitions of synthetic limits and colimits introduced in \Cref{colimitdef}. The notion of limits and colimits in Segal spaces we are considering are the ones presented in \cite{yonedarasekh}. The goal of this section is prove the next result:

\begin{theorem}
	The definitions of limits and colimits from \Cref{colimitdef}, under interpretation in the category of bisimplicial sets, coincide with the definitions of limits and colimits from \cite{yonedarasekh}.
\end{theorem}

Firstly, let us recall some standard notation and terminology we will use:

\begin{itemize}
\item $\Delta$ will denote the category whose objects are the non-empty linearly ordered finite sets, $[n] :\equiv \{0\leq 1 \leq \cdots \leq n \}$ with $n\in \N$, and morphisms the order preserving maps between linearly ordered sets.
\item The representable presheaf $\Delta^{op} \to \set$  given by $[n] \in \Delta$ is noted by $\Delta[n]$.
\item For each $ n\in \N $, $ F(n): \Delta^{op} \times \Delta^{op} \to \set  $ is defined as $$ F(n)_{k,l} :\equiv \Delta([k],[n]). $$ Two particular examples are $F(1)$ and $F(0)$, the last of which is the terminal object.
\item The category of simplicial sets is cartesian closed: for $X,Y \in \sset$, the \textbf{mapping simplicial set} is defined as \[Map(X,Y)_n :\equiv \sset(X\times \Delta[n],Y).\]
\item The category of bisimplicial sets is cartesian closed: for $X,Y \in \bisim$, the \textbf{mapping simplicial set} is defined as $$(Y^X)_{kl} :\equiv \bisim(F(k)\times \Delta[l] \times X,Y).$$
\item  Decorated arrows $``\twoheadrightarrow''$ indicate fibrations a the model structure. We will use this notation for the different models involved, but there is no room for confusion as the type of fibration is clear from the context.
\end{itemize}

Using \Cref{sttmodel} it is possible to conclude that simplicial type theory can be interpreted in bisimplicial sets. In particular $\homtyp$ types do coincide with $hom$ spaces. From \cite{rezk}, if $T$ is a Segal space then the mapping space between two objects $a,b\in T_0$ is constructed via the pullback square of spaces:

\begin{equation} \label{hompullback}
	\begin{tikzcd}
		map_T(a,b) \ar[r] \ar[d] \ar[dr,phantom,"\lrcorner"{near start}] & T_1 \ar[d, twoheadrightarrow] \\
		\Delta[0] \ar[r,"{(a,b)}"'] & T_0\times T_0.
	\end{tikzcd}
\end{equation}

On the other hand, from \Cref{theorem44} we get
\[\simplext\to A\simeq \sum\limits_{k:\partial\simplext\to A}\langle \simplext\to A|_k^{\partial\simplext}\rangle,\]
and it is immediate that the type of functions $\partial\simplext\to A$ is equivalent to $A\times A$. Therefore, if $A$ is a Segal type then $A^{\simplext}$ is the total space of a family over $A\times A$ whose fibers are $\homtyp$ types. This family is exactly $\homtyp (-,-):A\times A\to\universetyp$. Under the standard interpretation $\simplext$ is $F(1)$ and $\pointtyp$ is $F(0)$. Given $a,b:A$ we obtain $\homtyp_A(a,b)$ as the substitution along $(a,b):\pointtyp\to A\times A$ into $\homtyp_A(-,-)$. When we interpret our type theory into simplicial spaces this gives us the pullback square of bisimplicial sets
\[
\begin{tikzcd}
	hom_A(a,b) \ar[r] \ar[d,twoheadrightarrow] \ar[dr,phantom,"\lrcorner"{very near start}] & A^{F(1)} \ar[d, twoheadrightarrow] \\
	F(0) \ar[r,"{(a,b)}"'] & A\times A.
\end{tikzcd}
\]

We obtain the following comparison result:

\begin{proposition} \label{hom-eq-mapping}
  Let $ A $ be a Segal type and $ a,b:A $. The interpretation of $\homtyp_A(a,b)$ into bisimplicial sets is a constant simplicial space with value the Kan complex $ map_A(a,b) $.
\end{proposition}
\begin{proof}
 Note that $ hom_A(a,b)_0=map_A(a,b) $. Furthermore, from \cite[Proposition 8.29]{riehlshulman} $ \homtyp_A (-,-): A \to A \to \universetyp $ is a two-sided discrete fibration, so it is $ A^{F(1)} \to A \times A $. Hence, since each fiber $ \homtyp_A(a,b) $ is discrete it follows from \Cref{discreteconstant} that $ hom_A(a,b) $ is a constant simplicial space at $ map_A(a,b)$.
\end{proof}

The definition of a colimit from \cite{yonedarasekh} is formulated as follows: let $f:K\to A$ be a map of bisimplicial sets where $A$ is a Segal space. The \textbf{space of cocones under} $K$ is the Segal space
\[ A_{f/}:\equiv F(0)\times_{A^K}(A^{K})^{F(1)}\times_{A^K}A. \]
If $a\in A$ then we can construct the space of objects under $a$ by taking the map $a:F(0)\to A$ and applying the previous definition.

The map $f:K\to A$ has a colimit if the Segal space $A_{f/}$ has an initial object. This is to say that there exist an object $\sigma\in A_{f/}$ such that the map $(A_{f/})_{\sigma/} \to A_{f/}$ is a trivial Reedy fibration. Let $a\in A$ be an initial object. Therefore, for any object $x\in A$ we have the pullback:
\[
\begin{tikzcd}
	hom_A(a,x) \ar[r] \ar[d,twoheadrightarrow] \ar[rd,phantom,"\lrcorner"{near start}]& A_{a/} \ar[d, twoheadrightarrow] \\
	F(0) \ar[r,"x"] & A.
\end{tikzcd}
\]
The assumption that the map on the right is also an equivalence implies that $hom_A(a,x)$ is equivalent to $F(0)$. This is to say that $hom_A(a,x)$ is contractible and it means that up to homotopy there is a unique map from $a$ to $x$. Let us see first why this prospect of initiality matches with \Cref{definitioninitiality}. In \Cref{hom-eq-mapping} we showed that $hom_A(a,b)$ is a bisimplicial set constant at $map_A(a,b)$. It remains to show that equivalences are translated correctly. This follows immediately from \cite[Lemma 4.3]{shulmanreedy}:

\begin{lemma}
	For a map $f$ between fibrations $p_1:E_1\to B$ and $p_2:E_2\to B$
	the following are equivalent:
	
	\begin{enumerate}
		\item $f$ is a weak equivalence.
		\item $\isequiv_B(f)\to B$ has a section.
		\item There is a map $\pointtyp\to \prod_B \isequiv_B(f)$.
		\item $\isequiv_B(f)\to B$ is an acyclic fibration.
	\end{enumerate}
\end{lemma}

Theorem 4.4.3 and Theorem 4.2.6 of \cite{hottbook} imply that a function being an equivalence is logically equivalent to having contractible fibers. The object $\isequiv_B(f)$ encodes this last fact as is shown throughout \cite[Section 4]{shulmanreedy}. Therefore, if a type is contractible then under the standard interpretation it is also contractible in the model. We can conclude that our \Cref{definitioninitiality} of initiality coincides with the one in the semantics. All there is left to do is to compare the cocones with \Cref{conesdefinition}. It is useful to see that the space $A_{f/}$ is constructed with the successive pullbacks

\[
\begin{tikzcd}
	F(0)\times_{A^K}(A^{K})^{F(1)} \ar[r] \ar[d,twoheadrightarrow] \ar[rd,phantom,"\lrcorner"{near start}] & (A^{K})^{F(1)} \ar[d, twoheadrightarrow] \\
	F(0) \ar[r,"f"'] & A^K
\end{tikzcd}
\begin{tikzcd}
	A_{f/} \ar[r] \ar[d,twoheadrightarrow] \ar[rd,phantom,"\lrcorner"{near start}] & F(0)\times_{A^K}(A^{K})^{F(1)} \ar[d, twoheadrightarrow] \\
	A \ar[r,"\triangle"'] & A^K.
\end{tikzcd}
\]
Similar to the case with $\homtyp$ types: $(A^K)^{\simplext}$ is the total space of the type family $\homtyp_{A^K}(-,-)$ over $A^K\times A^K$. Then if we take $f:K\to A$ and any $x:A$ the type $\homtyp_{A^K}(f,\triangle x)$ is the substitution of $\homtyp_{A^K}(-,-)$ along $(f,\triangle):A \to A^K\times A^K$. Thus this is the pullback in the semantics:
\[
\begin{tikzcd}
	A_{f/} \ar[r] \ar[d] \ar[rd,phantom,"\lrcorner",very near start] & (A^X)^{F(1)} \ar[d] \\
	A\times F(0) \ar[r,"\triangle\times f"] & A^X \times A^X.
\end{tikzcd}
\]


	\bibliography{references-limits}
	\bibliographystyle{alpha}
	
\end{document}